\documentclass{article}

\usepackage{multicol}

\usepackage{amssymb,amsmath,amsthm}
\usepackage[]{algorithm2e}
\usepackage{tikz}

\newtheorem{definition}{Definition}[section]
\newtheorem{theorem}[definition]{Theorem}

\newtheorem{lemma}[definition]{Lemma}

\newtheorem{conjecture}[definition]{Conjecture}

\usepackage{ifxetex}
\ifxetex
  \usepackage{fontspec}
\else
  \usepackage[T1]{fontenc}
  \usepackage[utf8]{inputenc}
  \usepackage{lmodern}
\fi

\title{Counterexamples of the Bhattacharya-Friedland-Peled conjecture}

\date{September 13, 2021}

\author{
Yen-Jen Cheng
\thanks{Department of Applied Mathematics, National Yang Ming Chiao Tung University, Hsinchu 30010, Taiwan R.O.C.
{\tt Email: yjc7755@gmail.com}}
\and
Chia-An Liu
\thanks{Department of Mathematics, Soochow University, Taipei 111002, Taiwan R.O.C.
{\tt Email: liuchiaan8@gmail.com}}
\and
Chih-wen Weng
\thanks{Department of Applied Mathematics, National Yang Ming Chiao Tung University, Hsinchu 30010, Taiwan R.O.C.
{\tt Email: weng@math.nctu.edu.tw }}
}

\begin{document}
\maketitle

\begin{abstract}
The Brauldi-Hoffman conjecture, proved by Rowlinson in 1988, characterized the graph with maximal spectral radius among all simple graphs with prescribed number of edges. In 2008, Bhattacharya, Friedland, and Peled proposed an analog, which will be called the BFP conjecture in the following, of the Brauldi-Hoffman conjecture for the bipartite graphs with fixed numbers of edges in the graph and vertices in the bipartition. The BFP conjecture was proved to be correct if the number of edges is large enough by several authors. However, in this paper we provide some counterexamples of the BFP conjecture.
\end{abstract}



\section{Introduction and preliminaries}    \label{sec_introduction}

Let $G$ be a simple graph on $n$ vertices. The {\it adjacency matrix} $A=(a_{ij})$ of $G$ is a 0-1 square matrix of order $n$ with rows and columns indexed by the vertex set $V(G)$ of $G$ such that for two vertices $i,j\in V(G)$, $a_{ij}=1$ if and only if $i,j$ are adjacent in $G$. The {\it spectral radius} $\rho(G)$, or $\rho(A),$ of $G$ is the largest eigenvalue of the adjacency matrix $A$ of $G$. In 1976, Brauldi and Hoffman proposed the problem~\cite[p.438]{bfvs78} of finding the maximum spectral radius of a graph with exactly $e$ edges. One decade later in 1985, they gave the Brauldi-Hoffman conjecture in~\cite{bh85} stating that the maximum spectral radius of a graph with $e$ edges is attained by taking a complete graph and adding a new vertex which is adjacent to a corresponding number of vertices in the complete graph. This conjecture was proved in 1988 by Rowlinson~\cite{r88}. See~\cite{f88,s87} for the proof of some partial cases.

In 2008, a bipartite graphs analogue of the Brauldi-Hoffman conjecture was settled by Bhattacharya, Friedland, and Peled~\cite{bfp08} with the following statement: For a connected bipartite graph $G$ with $e$ edges, its spectral radius $\rho(G)\leq \sqrt{e}$, and  equality holds if and only if $G$ is a complete bipartite graph. Throughout the paper let $p, q, e$ be positive integers with $p\leq q$ and $e<pq$. Let $\mathcal{K}(p,q,e)$ be the family of subgraphs of $K_{p,q}$ with precisely $e$ edges and with no isolated vertices and {\it which are not complete bipartite graphs}, where $K_{p,q}$ is the complete graph with bipartition orders $p$ and $q$. The BFP conjecture was then given as follows~\cite[Conjecture~1.2]{bfp08}.
\begin{conjecture}  \label{conj_BFP}
Let $p,q,e$ be positive integers satisfying $e<pq$. An extremal graph that solves
$$\max_{G\in\mathcal{K}(p,q,e)}\rho(G)$$
is obtained from a complete bipartite graph by adding one vertex and a corresponding number of edges.
\end{conjecture}

Furthermore, in~\cite[Theorem~8.1]{bfp08}, the BFP conjecture was proved for the case that $e=st-1$ if the positive integers $s,t$ satisfy $2\leq s\leq p\leq t\leq q\leq t+\frac{t-1}{s-1}.$ They also verified that the only extremal graph is obtained from $K_{s,t}$ by deleting one edge. The BFP conjecture did not indicate that the adding vertex of the obtained extremal graph goes into which partite set. Hence, we define two families of bipartite graphs as follows:
  For  $e> pq-q$ (resp. $e> pq-p$), let $^eK_{p,q}$ (resp. $K^e_{p,q}$) denote the graph which is obtained from $K_{p,q}$ by deleting $pq-e$ edges which are incident on a common vertex in the partite set of order $p$ (resp. of order $q$). Then, the extremal graph described in the BFP conjecture is $G={}^eK_{s,t}$ or $G=K^e_{s,t}$ for some positive integers $s, t$ with $s\leq t$, where the equality notation $=$ is graph isomorphism. For example, one may see Figure 1 for the graphs $^4K_{2,3}$ and $K^5_{2,3}.$

\medskip

\begin{center}
\begin{multicols}{2}
\begin{picture}(45,80)
\put(10,40){\circle*{3}} \put(10,60){\circle*{3}}
\put(40,30){\circle*{3}} \put(40,50){\circle*{3}} \put(40,70){\circle*{3}}
\qbezier(10,60)(25,45)(40,30) \qbezier(10,60)(25,55)(40,50) \qbezier(10,60)(25,65)(40,70)
\qbezier(10,40)(25,55)(40,70)
\put(10,0){$^4K_{2,3}$}
\end{picture}

\begin{picture}(45,80)
\put(10,40){\circle*{3}} \put(10,60){\circle*{3}}
\put(40,30){\circle*{3}} \put(40,50){\circle*{3}} \put(40,70){\circle*{3}}
\qbezier(10,60)(25,45)(40,30) \qbezier(10,60)(25,55)(40,50) \qbezier(10,60)(25,65)(40,70)
\qbezier(10,40)(25,55)(40,70) \qbezier(10,40)(25,45)(40,50)
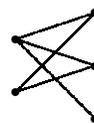
\put(-2,0){$K^5_{2,3}={}^5K_{2,3}$}
\end{picture}

\end{multicols}

\bigskip

\textbf{Figure 1.} The graphs $^4K_{2,3}$ and  $K^5_{2,3}$.
\end{center}
\bigskip

In 2010~\cite{cfksw10}, Chen {\it et al.} gave an affirmative answer to the BFP conjecture provided that $e=pq-2$ by comparing the spectral radii of three bipartite graphs obtained from $K_{p,q}$ by deleting two edges. Moreover, they refined the BFP conjecture under the assumption that $e\geq pq-p+1$ as follows~\cite[Conjecture~11]{cfksw10}.
\begin{conjecture}  \label{conj_Fu}
Let $p,q,e$ be positive integers satisfying $p\leq q$ and $pq-p<e<pq$. An extremal graph that solves
$$\max_{G\in\mathcal{K}(p,q,e)}\rho(G)$$
is obtained by $K^e_{p,q}.$
\end{conjecture}

The assumption $pq-p<e<pq$ above ensures that every graph in $\mathcal{K}(p,q,e)$ including $^eK_{p,q}$ and $K^e_{p,q}$ is connected.  Conjecture~\ref{conj_Fu} was proved by Liu and Weng~\cite{lw15} in 2015. For applications, there are extending results on the spectral characterization of the nearly complete bipartite graphs~\cite{cfw18,lw17}.
However, as $e\leq pq-p$, things have changed. Let $K^\pm_{p,q}$ denote the bipartite graph obtained from $K_{p,q}$ by deleting an edge $e_1$, and then adding an edge $e_2$, not incident with the previously deleted edge $e_1$,
joining a new vertex and a vertex in the partite set of order $p$.  For example, see Figure 2 for the graphs $K^\pm_{2,3}$ and $K^\pm_{3,3}$.
\medskip

\begin{center}
\begin{multicols}{2}

\begin{picture}(45,80)
\put(10,40){\circle*{3}} \put(10,60){\circle*{3}}
\put(40,10){\circle*{3}}
\put(40,30){\circle*{3}} \put(40,50){\circle*{3}} \put(40,70){\circle*{3}}
\qbezier(10,60)(25,45)(40,30) \qbezier(10,60)(25,55)(40,50) \qbezier(10,60)(25,65)(40,70) \qbezier(10,60)(25,35)(40,10)
\qbezier(10,40)(25,55)(40,70) \qbezier(10,40)(25,45)(40,50)
\put(-2,-7){$K^\pm_{2,3}={}^6K_{2,4}$}
\end{picture}

\begin{picture}(45,80)
 \put(10,20){\circle*{3}}
\put(10,40){\circle*{3}} \put(10,60){\circle*{3}}
\put(40,10){\circle*{3}}
\put(40,30){\circle*{3}} \put(40,50){\circle*{3}} \put(40,70){\circle*{3}}
\qbezier(10,40)(25,35)(40,30)
\qbezier(10,20)(25,45)(40,70)
\qbezier(10,20)(25,35)(40,50)
\qbezier(10,60)(25,45)(40,30) \qbezier(10,60)(25,55)(40,50) \qbezier(10,60)(25,65)(40,70) \qbezier(10,60)(25,35)(40,10)
\qbezier(10,40)(25,55)(40,70) \qbezier(10,40)(25,45)(40,50)
\put(13,-7){$K^\pm_{3,3}$}
\end{picture}
\end{multicols}

\bigskip

\textbf{Figure 2.} The graphs $K^\pm_{2,3}$ and $K_{3, 3}^\pm$.

\end{center}
\bigskip

 We will show in Theorem~\ref{thm_e_p} of the next section that the BFP conjecture indeed fails with counterexample graphs of the form $K^\pm_{p,q-k}\in \mathcal{K}(p, q, e)$ under some restrictions on the positive integers $p, q, e$ and a nonnegative integer $k$.

\section{Counterexamples of the BFP conjecture}\label{sec_main}

Let $D=(d_1,d_2,\ldots,d_p)$ be a nonincreasing sequence of positive integers in which $d_1=q>d_p$ and $e=d_1+d_2+\cdots+d_p$. Then, a simple bipartite graph $G_D\in \mathcal{K}(p,q,e)$ is obtained as follows. Let $G_D$ denote the bipartite graph with bipartition $X\cup Y$, where $X=\{x_1,x_2,\ldots,x_p\},$ $Y=\{y_1,y_2,\ldots,y_q\}$, and $x_iy_j$ is an edge if and only if $j\leq d_i.$ The sequence $D$ also defines a 0-1 Ferrers diagram $F(D)$ that has $p$ rows and $q$ columns in which the $i$-th row of $F(D)$ is composed of $d_i$ 1's on the left and $(q-d_i)$ 0's on the right, for $i=1,2,\ldots,p.$ For example, if $D=(5,3,1,1)$ then the bipartite graph $G_{D}$ and its associated Ferrers diagram are shown in Figure 2.

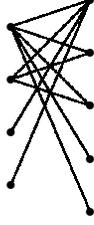
\begin{center}
\begin{multicols}{2}
\begin{picture}(50,95)
\put(10,20){\circle*{3}} \put(10,40){\circle*{3}} \put(10,60){\circle*{3}} \put(10,80){\circle*{3}}
\put(40,10){\circle*{3}} \put(40,30){\circle*{3}} \put(40,50){\circle*{3}} \put(40,70){\circle*{3}} \put(40,90){\circle*{3}}
\qbezier(10,80)(25,85)(40,90) \qbezier(10,80)(25,75)(40,70) \qbezier(10,80)(25,65)(40,50) \qbezier(10,80)(25,55)(40,30) \qbezier(10,80)(25,45)(40,10)
\qbezier(10,60)(25,75)(40,90) \qbezier(10,60)(25,65)(40,70) \qbezier(10,60)(25,55)(40,50)
\qbezier(10,40)(25,65)(40,90)
\qbezier(10,20)(25,55)(40,90)
\end{picture}

\begin{equation}
F(5,3,1,1)=
\begin{array}{ccccc}
1 & 1 & 1 & 1 & 1  \\
1 & 1 & 1 & 0 & 0 \\
1 & 0 & 0 & 0 & 0 \\
1 & 0 & 0 & 0 & 0
\end{array}
\nonumber
\end{equation}
\end{multicols}
\textbf{Figure 2.} The graph $G_{(5,3,1,1)}$ and its associated Ferrers diagram.
\end{center}

The adjacency matrix $A$ of the bipartite graph $G_D$ with nonincreasing degree sequence $D=(d_1,d_2,\ldots,d_p)$ is
\begin{equation}    \label{eq_A}
A=\left(\begin{array}{cc}
    O_{p\times p} & F(D)  \\
    F(D)^T & O_{q\times q}
    \end{array}\right),
\end{equation}
where $O_{m\times n}$ is the $m$-by-$n$ zero matrix. Let $H(D):=F(D)F(D)^T$, which is the $p\times p$ matrix as follows:
\begin{equation}\label{H}
H(D)=(\min\{d_i,d_j\})_{1\leq i,j\leq p}=\left(\begin{array}{ccccc}
    d_1 & d_2 & d_3 & \cdots & d_p  \\
    d_2 & d_2 & d_3 & \cdots & d_p  \\
    d_3 & d_3 & d_3 & \cdots & d_p  \\
    \vdots & \vdots & \vdots & \ddots & \vdots    \\
    d_p & d_p & d_p & \cdots & d_p
    \end{array}\right).\end{equation}

\begin{lemma}   \label{lem_rhoA_rhoHD}
If $D=(d_1,d_2,\ldots,d_p)$ is a nonincreasing sequence of positive integers, then the spectral radius of  graph $G_D$ is
$$\rho(G_D)=\rho(H(D))^{1/2}.$$
\end{lemma}
\begin{proof} Let $A$ be the adjacency matrix of $G_D$.
By \eqref{eq_A}, we have
$$A^2=\left(\begin{array}{cc}
    F(D)F(D)^T & O_{p\times q}  \\
    O_{q\times p} & F(D)^TF(D)
    \end{array}\right).$$
Since $F(D)F(D)^T$ and $F(D)^TF(D)$ have the same nonzero eigenvalues and by using (\ref{H}),  we have
$$\rho^2(G_D)=\rho(A^2)=\rho(F(D)F(D)^T)=\rho(H(D)),$$
and the result follows.
\end{proof}

If there are consecutive repeated terms in a sequence, we write them in exponential forms with square brackets on the powers. For example, the sequence $(5,3,1,1)$ will be written as $(5,3,1^{[2]}).$ Let $p, q, k, e$ be positive integers such that $p>2$, $q>kp+2$ and $e=p(q-k)$.  Let
\begin{equation}    \label{eq_D*}
D_k^*:=(q-k+1,(q-k)^{[p-2]},q-k-1).
\end{equation}
Then $G_{D_k^*}=K^\pm_{p,q-k},$ and the associated $H$ matrix of $D_k^*$ in (\ref{H}) is the following $p\times p$ matrix
\begin{equation}\label{HDk*}
H(D_k^*)=\left(\begin{array}{ccccc}
    q-k+1 & q-k & \cdots & q-k & q-k-1  \\
    q-k & q-k & \cdots & q-k & q-k-1  \\
    \vdots & \vdots & \ddots & \vdots & \vdots  \\
    q-k & q-k & \cdots & q-k & q-k-1    \\
    q-k-1 & q-k-1 & \cdots & q-k-1 & q-k-1
    \end{array}\right).\end{equation}
We will investigate $\rho(H(D_k^*))$. Assume that $M$ is a positive (that is, each entry is positive) symmetric matrix in the following block form
$$M=\left(\begin{array}{ccc}
    M_{1,1} & \cdots & M_{1,m}    \\
    \vdots & \ddots & \vdots   \\
    M_{m,1} & \cdots & M_{m,m}
    \end{array}\right)$$
according to a partition $\{X_1,\ldots,X_m\}$ of its row (column) indices, where the diagonal blocks $M_{i,i}$ are square matrices of orders $|X_i|$ for $i=1,\ldots,m.$ Let $b_{ij}$ be the sum of entries of $M_{i,j}$ divided by the number of its rows, {\it i.e.} the average row-sums of $M_{i,j}.$ Then $B=(b_{ij})$ is called a {\it quotient matrix} of $M$. Additionally, if $M_{i,j}$ has a constant row-sum for every $1\leq i,j\leq m$, then $B$ is called an {\it equitable quotient matrix} of $M$. The following lemma is straightforward from the matrix multiplication, and its details can be found in~\cite[Chapter~2]{bh11} and~\cite[Chapter~9]{gr01}. To make this paper self-contained, we give a proof to the lemma.
\begin{lemma}\label{lem_eq_q_mat}
If $B$ is an equitable quotient matrix of a positive symmetric matrix $M$, then every eigenvalue of $B$ is an eigenvalue of $M$. Moreover the spectral radius $$\rho(B)=\rho(M).$$
\end{lemma}
\begin{proof}
Suppose that $B$ is an $m\times m$ equitable quotient matrix of the $n\times n$ matrix $M$ according to the partition $\{X_1,\ldots,X_m\}$. The {\it characteristic matrix} $S=(s_{ij})$ is the $n\times m$ matrix whose $j$-th column is a 0-1 vector with $s_{ij}=1$ if and only if $i\in X_j,$ for $i=1,2, \ldots, n$ and $j=1,2, \ldots,m.$ Let $v$ be a positive eigenvector of $B$ corresponding to the eigenvalue $\lambda(B)$, i.e.,   $Bv=\lambda(B)v$. Then
$$M(Sv)=(MS)v=(SB)v=S(Bv)=S(\lambda(B)v)=\lambda(B)(Sv),$$
and hence $\lambda(B)$ is an eigenvalue of $M.$ Moreover, since $M$ and $B$ are positive, they have unique  positive eigenvectors, up to scalar product, corresponding to their largest eigenvalues $\rho(M)$ and $\rho(B)$ respectively. If the above $v$ is chosen to be an eigenvector of $B$ corresponding to $\rho(B)$, then $\rho(B)$ is an eigenvalue of $M$ with positive eigenvector $Sv$. Hence $\rho(M)=\rho(B)$.
\end{proof}

\begin{lemma}   \label{lem_cubic_poly}
If $p, q, k, e$ are positive integers such that $p>2$, $q>kp+2$ and $e=p(q-k)$,
then $\rho(K_{p, q-k}^\pm)^2$ is the largest root of  the cubic polynomial
\begin{equation}    \label{eq_char_poly_D*}
g(x)=x^3-ex^2+\big((2q-2k-1)(p-1)-1\big)x-(p-2)(q-k-1).
\end{equation}
\end{lemma}

\begin{proof} As mentioned above $K_{p, q-k}^\pm=G_{D_k^*}$, where $D_k^*$ is the sequence defined in (\ref{eq_D*}).
According to the partition $\Pi=\{\{1\},\{2,\ldots,p-1\},\{p\}\}$, $H(D_k^*)$ in (\ref{HDk*}) has the equitable quotient matrix
$$\Pi(H(D_k^*))=\left(\begin{array}{ccc}
    q-k+1 & (p-2)(q-k) & q-k-1  \\
    q-k & (p-2)(q-k) & q-k-1    \\
    q-k-1 & (p-2)(q-k-1) & q-k-1    \\
    \end{array}\right).$$
By direct computation, we find $g(x)$ to be the characteristic polynomial of $\Pi(H(D_k^*))$. By Lemma~\ref{lem_rhoA_rhoHD} and Lemma~\ref{lem_eq_q_mat}, every root of $g(x)$ is the square of an eigenvalue of $K_{p, q-k}^\pm$,
and $\rho(K_{p, q-k}^\pm)^2$ is the largest root of $g(x)$.
\end{proof}

We shall compare values  $\rho(K_{p, q-k}^\pm)$ and $\rho(G)$ for some expected graphs $G$ in the BFP conjecture. We determine
such $G$ in the following lemma.

\begin{lemma}\label{l2.5}
 If $p, q, k, e$ are positive integers such that $p>2$, $q>kp+2$, and $e=p(q-k)$, then
 the graphs in $\mathcal{K}(p, q, e)$ obtained from a complete bipartite graph by adding one vertex and corresponding number of edges are exactly the following $k$ graphs:
 $${}^eK_{p,q-a}=G_{D_{k, a}},$$
where
\begin{equation}\label{Dka}D_{k,a}=((q-a)^{[p-1]},q-a-(k-a)p)\end{equation}
for $a=0,1,\ldots,k-1$.
 \end{lemma}

 \begin{proof} The desired graphs are $^eK_{s, t}$ or $K_{s, t}^e$ in $\mathcal{K}(p, q, e)$ with $s\leq t$.
If $s<p$ then $pk=pq-e>pq-st\geq pq-(p-1)q=q,$ a contradiction to $q>kp+2$. Hence $s=p$.
If $K_{p, t}^e\in \mathcal{K}(p, q, e)$, then from the definition we have  $p>pt-e>0$, and clearly
$pt-e=pt-p(q-k)=p(t-q+k)$, implying $1>t-q+k>0$, a contradiction to the fact that $t-q+k$ is an integer.
 Hence the remain cases are $^eK_{p, q-a}\in \mathcal{K}(p, q, e)$, where  $a:=q-t$. We need to have $a<k$ since  $pq-kp=e<p(q-a).$
 Indeed  each $^eK_{p, q-a}$ exists in $\mathcal{K}(p, q, e)$ for $a=0, 1, \ldots, k-1$.  The graph ${}^eK_{p,q-a}$ is clear to be $G_{D_{k, a}}$ as stated in the statement.
 \end{proof}

The above lemma indicates that the graphs of type $K^e_{s, t}$ might not exist  in $\mathcal{K}(p, q, e)$ under some special restrictions of $p, q, e$, so there is no hope to further extend Conjecture~\ref{conj_Fu}. Now we focus on graphs of the type  ${}^eK_{s, t}$.

\begin{lemma}\label{l2.6}
If $p,q,k$ are positive integers satisfying $p>2,$ $q>kp+2,$ and $e=p(q-k),$  then
\begin{equation}\label{comp}\rho(K_{p, q-k}^\pm)>\rho({}^eK_{p,q-a})\quad \text{for}\quad a=0,1,\ldots,k-1.\end{equation}
\end{lemma}
\begin{proof}
As mentioned above ${}^eK_{p,q-a}=G_{D_{k, a}}$, where $D_{k,a}$ is the sequence defined in (\ref{Dka}).
 Let $H(D_{k, a})$ be defined as in  (\ref{H}) with $D=D_{k, a}$.
  According to the partition $\Pi=\{\{1\}, \{2,\ldots,p-1\},\{p\}\}$, $H(D_{k, a})$ has  equitable quotient matrix
 $$\Pi(H(D_{k, a}))=\left(\begin{array}{ccc}
    q-a & (p-2)(q-a) &  q-a-(k-a)p  \\
    q-a & (p-2)(q-a) &  q-a-(k-a)p  \\
    q-a-(k-a)p & (p-2) (q-a-(k-a)p) &  q-a-(k-a)p   \\
    \end{array}\right).$$
The characteristic polynomial of $\Pi(H(D_{k, a}))$ is
$$f(x)=x^3-ex^2+(k-a)p(p-1)(q-a-(k-a)p)x,$$
whose largest root is $\rho({}^eK_{p,q-a})^2$ by Lemma~\ref{lem_rhoA_rhoHD} and Lemma~\ref{lem_eq_q_mat}.
Notice that the least value of the coefficient of $x$ in $f(x)$ among $a\in [0, k-1]$ is attained when either $a=0$ or $a=k-1$,
since this coefficient is a quadratic polynomial in $a$ with leading coefficient $-p(p-1)^2$, a negative number.
On the other hand, recall that $\rho(K_{p, q-k}^\pm)^2$ is the largest root of  $g(x)$ in~\eqref{eq_char_poly_D*} by Lemma~\ref{lem_cubic_poly}. The difference $f(x)-g(x)$ of $f(x)$ and $g(x)$ is
$$\big((p-1)[p(k-a)(q-a-(k-a)p)-2q+2k+1]+1\big)x+(p-2)(q-k-1),$$
whose constant term $(p-2)(q-k-1)$ in $f(x)-g(x)$ is positive from the assumptions.  The coefficient of $x$ in $f(x)-g(x)$ takes the least value in one of the following two positive values:
$$\left\{\begin{array}{ll}
(p-1)[(kp-2)(q-kp-2)+2k-3]+1,& \hbox{if $a=0$,} \\
(p-1)[(p-2)(q-p-k-2)+p-3]+1 & \hbox{if $a=k-1>0$}.
\end{array}
\right.$$
Here we use  $pk> p+k$ if $p\geq 3$ and $k\geq 2$ to ensure $q-p-k-2>q-pk-2>0$ in the case $a=k-1>0$.
Thus $f(x)>g(x)$ for all $x>0$. In particular, for $x\geq \rho(K_{p, q-k}^\pm)^2$,  $f(x)>g(x)\geq0=g(\rho(K_{p, q-k}^\pm)^2)$, which implies
that the largest root $\rho({}^eK_{p,q-a})^2$ of $f(x)$ is less than $\rho(K_{p, q-k}^\pm)^2$, and
(\ref{comp}) follows.
\end{proof}

\begin{theorem}   \label{thm_e_p}
If $p,q,k$ are positive integers satisfying $p>2,$ $q>kp+2,$ and $e=p(q-k)$, then Conjecture~\ref{conj_BFP} is false.
\end{theorem}
\begin{proof}
The theorem immediately follows from Lemma~\ref{l2.5} and Lemma~\ref{l2.6}.

\end{proof}

\section*{Acknowledgments}
This research is supported by the Ministry of Science and Technology of Taiwan R.O.C. under the projects MOST 110-2811-M-A49-505, MOST 109-2115-M-031-006-MY2, and MOST 109-2115-M-009-007-MY2.

\end{document}